\pdfoutput=1
\RequirePackage{ifpdf}
\ifpdf 
\documentclass[pdftex]{sigma}
\else
\documentclass{sigma}
\fi

\numberwithin{equation}{section}

\newtheorem{Theorem}{Theorem}[section]
\newtheorem{Corollary}[Theorem]{Corollary}
\newtheorem{Lemma}[Theorem]{Lemma}

\begin{document}

\allowdisplaybreaks

\newcommand{\arXivNumber}{1701.08960}

\renewcommand{\thefootnote}{}

\renewcommand{\PaperNumber}{037}

\FirstPageHeading

\ShortArticleName{Gustafson--Rakha-Type Elliptic Hypergeometric Series}

\ArticleName{Gustafson--Rakha-Type Elliptic Hypergeometric\\ Series\footnote{This paper is a~contribution to the Special Issue on Elliptic Hypergeometric Functions and Their Applications. The full collection is available at \href{https://www.emis.de/journals/SIGMA/EHF2017.html}{https://www.emis.de/journals/SIGMA/EHF2017.html}}}

\Author{Hjalmar ROSENGREN}

\AuthorNameForHeading{H.~Rosengren}

\Address{Department of Mathematical Sciences, Chalmers University of Technology and\\ University of Gothenburg, SE-412 96 Gothenburg, Sweden} \Email{\href{mailto:hjalmar@chalmers.se}{hjalmar@chalmers.se}}
\URLaddress{\url{http://www.math.chalmers.se/~hjalmar/}}

\ArticleDates{Received February 02, 2017, in f\/inal form May 29, 2017; Published online June 01, 2017}

\Abstract{We prove a multivariable elliptic extension of Jackson's summation formula conjectured by Spiridonov. The trigonometric limit case of this result is due to Gustafson and Rakha. As applications, we obtain two further multivariable elliptic Jackson summations and two multivariable elliptic Bailey transformations. The latter four results are all new even in the trigonometric case.}


\Classification{33D67}

\renewcommand{\thefootnote}{\arabic{footnote}}
\setcounter{footnote}{0}

\section{Introduction}

By combining two integral evaluations previously obtained by Gustafson \cite{gu},
Gustafson and Rakha \cite{gr} evaluated the basic hypergeometric integral
\begin{gather}\label{grb}\int\frac{\prod\limits_{1\leq i< j\leq n}(z_i/z_j)_\infty(z_j/z_i)_\infty\prod\limits_{j=1}^n(S/z_j)_\infty}
{\prod\limits_{1\leq i<j\leq n}(tz_iz_j)_\infty\prod\limits_{i=1}^n\left(\prod\limits_{j=1}^3(c_jz_i)_\infty\prod\limits_{j=1}^n(d_j/z_i)_\infty\right)}\frac{dz_1}{z_1}\dotsm\frac{dz_{n-1}}{z_{n-1}}, \end{gather}
where the integration is over $|z_1|=\dots=|z_{n-1}|=1$, $(z)_\infty=\prod\limits_{j=0}^\infty(1-q^j z)$, the parameters satisfy
\begin{gather*}|q|, |t|, |c_1|, |c_2|, |c_3|, |d_1|, \dots, |d_n|<1, \end{gather*}
$z_n$ is determined from the integration variables through $z_1\dotsm z_n=1$ and $S=t^{n-2}c_1c_2c_3d_1\dotsm d_n$. By applying residue calculus to~\eqref{grb}, they could evaluate a certain multivariable basic hypergeometric f\/inite sum, equivalent to the case $p=0$ of Theorem~\ref{grt} below.

Since the seminal work of Date et al.~\cite{d} and Frenkel and Turaev~\cite{ft}, it has been recognized that basic hypergeometric functions appear as the trigonometric limit of more general elliptic hypergeometric functions. Elliptic extensions of Gustafson's two integral evaluations mentioned above were conjectured in \cite{sp, ds} and proved in~\cite{ra}. Spiridonov~\cite{sp} used these (at the time conjectural) evaluations to obtain an elliptic extension of \eqref{grb}. He also stated the corresponding summation formula as a conjecture. Although it seems likely that this conjecture can be deduced from Spiridonov's integral, such a derivation is still missing from the literature. The purpose of the present paper is to give a direct proof of Spiridonov's conjectured summation and to apply it to derive some further summation and transformation formulas.

It is worth mentioning that Spiridonov's elliptic extension of~\eqref{grb} can be interpreted as the identity between superconformal indices of two dual quantum f\/ield theories \cite[Sections~12.1.2--12.1.3]{sv}. This indicates that~\eqref{grb} and related results are not mere curiosities and that it is not unreasonable to expect further applications.

The plan of the paper is as follows. In Section~\ref{ms}, we prove Spiridonov's conjecture. The proof is elementary and provides in particular a signif\/icant simplif\/ication of the trigonometric case. The only previously known proof of the Gustafson--Rakha summation is the original one, which as we recall is based on f\/irst proving two auxiliary multiple integral evaluations, combining them to obtain~\eqref{grb} and f\/inally on a technical computation to pass from integrals to f\/inite residue sums. In Section~\ref{ts}, we give some applications of our result. Namely, combining the elliptic Gustafson--Rakha sum with a summation from~\cite{rs}, we obtain two transformation formulas and two further summation formulas for multivariable elliptic hypergeometric series. These four results are all new even in the trigonometric case.

{\bf Note added in proof:} After completing this work, I learned from Masahiko Ito and Masatoshi Noumi that they have independently proved Theorem~\ref{grt}, using a dif\/ferent method.

\section{Preliminaries}

When $z=(z_1,\dots,z_n)$ is a vector we will write $|z|=z_1+\dots+z_n$ and $Z=z_1\dotsm z_n$.

Throughout, $p$ and $q$ will be f\/ixed parameters with $|p|<1$. We employ the standard notation
\begin{gather*}\theta(z)=\prod\limits_{j=0}^\infty\big(1-p^jz\big)\big(1-p^{j+1}/z\big), \\
(z)_k=\begin{cases}\theta(z)\theta(qz)\dotsm\theta\big(q^{k-1}z\big), & k\in\mathbb Z_{\geq 0},\\
1/\theta\big(q^kz\big)\theta\big(q^{k+1}z\big)\dotsm\theta\big(q^{-1}z\big), & k\in\mathbb Z_{<0}\end{cases} \end{gather*}
as well as
\begin{gather*}\theta(z_1,\dots,z_m)=\theta(z_1)\dotsm\theta(z_m), \qquad (z_1,\dots,z_m)_k=(z_1)_k\dotsm(z_m)_k.\end{gather*}
Most of our computations are based on the elementary identities
\begin{gather*}\theta(1/z)=\theta(pz)=-\theta(z)/z, \\
(a)_{n+k}=(a)_n(aq^n)_k,\qquad (a)_{n-k}=(-1)^kq^{\binom k2}(q^{1-n}/a)^k\frac{(a)_n}{(q^{1-n}/a)_k}, \end{gather*}
which will be used without comment. All our sums contain the $A$-type factor
\begin{gather*}\frac{\Delta(zq^x)}{\Delta(z)}=\prod\limits_{1\leq i<j\leq n}\frac{q^{x_i}\theta(q^{x_j-x_i}z_j/z_i)}{\theta(z_j/z_i)}, \end{gather*}
where $z=(z_1,\dots,z_n)$ and $x=(x_1,\dots,x_n)$ is the summation index. We mention the useful identity \cite[equation~(3.8)]{r}
\begin{gather}\label{etf}
\frac{\Delta(z q^x)}{\Delta(z)}=(-1)^{|x|}q^{-\binom{|x|}2-|x|}\prod\limits_{i,j=1}^n\frac{(qz_i/z_j)_{x_i}}{(q^{-x_j}z_i/z_j)_{x_i}}
\end{gather}
and \cite[Example~20.53.3]{ww}
\begin{gather}\label{tpf}
\sum_{k=1}^n\frac{\prod\limits_{j=1}^{n+1}\theta(z_k/b_j)}{\theta(z_k/t)\prod\limits_{j=1,\, j\neq k}^n\theta(z_k/z_j)}=\frac{\prod\limits_{j=1}^{n+1}\theta(b_j/t)}{\prod\limits_{j=1}^n\theta(z_j/t)},
\end{gather}
valid for $tz_1\dotsm z_n=b_1\dotsm b_{n+1}$.

In the one-variable case, the most fundamental results for elliptic hypergeometric series are the
elliptic Jackson (or Frenkel--Turaev) summation
\begin{gather}\label{js}\sum_{x=0}^N\frac{\theta\big(aq^{2x}\big)}{\theta(a)}\frac{\big(a,b,c,d,e,q^{-N}\big)_x q^x}{\big(q,aq/b,aq/c,aq/d,aq/e,aq^{N+1}\big)_x}=\frac{(aq,aq/bc,aq/bd,aq/cd)_N}{(aq/b,aq/c,aq/d,aq/bcd)_N},\end{gather}
valid for $a^2q^{N+1}=bcde$, and the elliptic Bailey transformation
\begin{gather}\label{bat}\sum_{x=0}^N\frac{\theta\big(aq^{2x}\big)}{\theta(a)}\frac{\big(a,b,c,d,e,f,g,q^{-N}\big)_x q^x}{\big(q,aq/b,aq/c,aq/d,aq/e,aq/f,aq/g,aq^{N+1}\big)_x}\\
=\frac{(aq,aq/ef,\lambda q/e,\lambda q/f)_N}{(\lambda q,\lambda q/ef,aq/e,aq/f)_N} \sum_{x=0}^N\frac{\theta\big(\lambda q^{2x}\big)}{\theta(\lambda)}\frac{\big(\lambda,\lambda b/a,\lambda c/a,\lambda d/a,e,f,g,q^{-N}\big)_x\,q^x}{\big(q,aq/b,aq/c,aq/d,\lambda q/e,\lambda q/f,\lambda q/g,\lambda q^{N+1}\big)_x}\nonumber
\end{gather}
valid for $a^3q^{N+2}=bcdefg$, $\lambda =a^2q/bcd$ \cite{ft}. Numerous multivariable extensions of~\eqref{js} and~\eqref{bat} are known, see, e.g.,~\cite{cg,rbc,rw,r,rs1,rs,sw,w}; further examples are obtained in the present paper. We will need one such result, a multivariable extension of \eqref{js} obtained in~\cite{rs} (see~\cite{sc} for the case $p=0$). Namely, for $a^2q^{|N|+1}=bcde$,
\begin{gather}
\sum_{x_1,\dots,x_n=0}^{N_1,\dots,N_n}\frac{\Delta(zq^x)}{\Delta(z)}\frac{\theta\big(aq^{2|x|}\big)} {\theta(a)}\frac{(a,b,c)_{|x|}\prod\limits_{i=1}^n(d/z_i)_{|x|}}{\big(aq/b,aq/c,aq^{|N|+1}\big)_{|x|}
\prod\limits_{i=1}^n\big(aq^{|N|+1-N_i}/ez_i\big)_{|x|}}q^{|x|}\nonumber\\
\qquad\quad{}\times\prod\limits_{i=1}^n\frac{\big(aq^{|N|+1}/ez_i\big)_{|x|-x_i}(ez_i)_{x_i}\prod\limits_{j=1}^n\big(q^{-N_j}z_i/z_j\big)_{x_i}}
{(d/z_i)_{|x|-x_i}(aqz_i/d)_{x_i}\prod\limits_{j=1}^n(qz_i/z_j)_{x_i}}\nonumber\\
\qquad{}=\frac{(aq,aq/bc)_{|N|}}{(aq/b,aq/c)_{|N|}}\prod\limits_{i=1}^n\frac{(aqz_i/bd,aqz_i/cd)_{N_i}}{(aqz_i/d,aqz_i/bcd)_{N_i}}.\label{rsi}
\end{gather}

\section{The elliptic Gustafson--Rakha summation}\label{ms}

Our main result is the following identity. It is easy to see that the case $p=0$ is equivalent to \cite[Theorem~1.2]{gr} and the general case to the conjecture of \cite[p.~953]{sp}. Recall the notation $Z=z_1\dotsm z_n$.

\begin{Theorem}\label{grt}
For parameters subject to $q^{N-1}b_1\dotsm b_4z_1^2\dotsm z_n^2=1$,
\begin{gather}\sum_{\substack{x_1,\dots,x_n\geq 0,\\x_1+\dots+x_n=N}}
\frac{\Delta(zq^x)}{\Delta(z)}\prod\limits_{1\leq i<j\leq n}q^{x_ix_j}(z_iz_j)_{x_i+x_j}\prod\limits_{i=1}^n
\frac{\prod\limits_{j=1}^4(z_ib_j)_{x_i}}{z_i^{x_i}\prod\limits_{j=1}^n(qz_i/z_j)_{x_i}}\nonumber\\
\qquad{} =\begin{cases}
\displaystyle\frac{(Zb_1,Zb_2,Zb_3,Zb_4)_N}{Z^N(q)_N}, & n~\text{odd},\\[4mm]
\displaystyle\frac{(Z,Zb_1b_2,Zb_1b_3,Zb_1b_4)_N}{(Zb_1)^N(q)_N}, & n~\text{even}.
\end{cases}\label{gri}
\end{gather}
\end{Theorem}

We will prove Theorem \ref{grt} by induction on $N$. In the case $N=1$, we have $x_i=\delta_{ik}$ for some~$k$. Using $k$ as summation index,
Theorem~\ref{grt} reduces to the following theta function identity.

\begin{Lemma}
For parameters subject to $b_1b_2b_3b_4z_1^2\dotsm z_n^2=1$,
\begin{gather}\label{ib}\sum_{k=1}^n\frac{\prod\limits_{j=1}^4\theta(z_kb_j)}{z_k}\prod\limits_{j=1,\,j\neq k}^n\frac{\theta(z_kz_j)}{\theta(z_k/z_j)}
=\begin{cases}
\theta(Zb_1,Zb_2,Zb_3,Zb_4)/Z, & n~\text{odd},\\
\theta(Z,Zb_1b_2,Zb_1b_3,Zb_1b_4)/Zb_1, & n~\text{even}.
 \end{cases} \end{gather}
\end{Lemma}

\begin{proof}
We apply induction on $n$, starting from the trivial case $n=1$. Let $b_1=vw$ and $b_2=v/w$, with $v$ and $w$ free parameters. As a function of $w$, each term in the sum, as well as the right-hand side, has the form $f(w)=C\theta(aw,a/w)$ with $C$ and $a$ independent of $w$. It is a~classical fact that any such function is determined by its values at two generic points. Indeed, Weierstrass' identity (which is equivalent to the case $n=2$ of~\eqref{tpf}) states that
\begin{gather*}f(w)=f(b)\frac{\theta(cw,c/w)}{\theta(cb,c/b)}+f(c)\frac{\theta(bw,b/w)}{\theta(bc,b/c)}, \end{gather*}
provided that $bc, b/c\notin p^{\mathbb Z}$. Thus, it suf\/f\/ices to verify~\eqref{ib} for two independent values of~$b_1$. Assuming $n\geq 2$, we choose $b_1=1/z_{n-1}$ and $b_1=1/z_n$. By symmetry, it is enough to consider the second case. Then, the term corresponding to $k=n$ cancels and we are reduced to an identity equivalent to \eqref{ib}, with $n$ replaced by $n-1$ and $b_1$ by $z_n$.
\end{proof}

We mention that it is not hard to deduce \eqref{ib} from classical theta function identities. Indeed, let $t=b_{n+1}$ in \eqref{tpf} (so that the right-hand side vanishes) and then make the substitutions \smash{$n\mapsto n+4$}, $(z_1,\dots,z_n)\mapsto (z_1,\dots,z_n,1,-1,\sqrt p,-1/\sqrt p)$, $(b_1,\dots,b_n)\mapsto(z_1^{-1},\dots,z_n^{-1},b_1^{-1},b_2^{-1}$, $b_3^{-1},b_4^{-1})$. Using the elementary identities $\theta(z^2)=\theta(z,-z,\sqrt pz,-\sqrt pz)$ and $2=\theta(-1,\sqrt p,-\sqrt p)$, one may deduce that the left-hand side of~\eqref{ib} is equal to
\begin{gather*}\frac 1 2\left({(-1)^{n+1}Z\prod\limits_{j=1}^4\theta(b_j)}+{Z\prod\limits_{j=1}^4\theta(-b_j)}+\frac 1{\sqrt p} {\prod\limits_{j=1}^4\theta\big(\sqrt pb_j\big)}
-\frac 1{\sqrt p} {\prod\limits_{j=1}^4\theta\big({-}\sqrt pb_j\big)}\right).
\end{gather*}
The fact that this equals the right-hand side of~\eqref{ib} follows from Jacobi's fundamental formulae \cite[Section~21.22]{ww}.

The inductive step in the proof of Theorem \ref{grt} is almost identical to that of \cite[Theorem~5.1]{r}. Denoting the right-hand side of~\eqref{gri} by
$R_N(Z;b_1,b_2,b_3,b_4)$ (where we for a moment consider~$Z$ as a free variable) we observe that, regardless of the parity of~$n$,
\begin{gather*}R_{N+1}(Z;b_1,b_2,b_3,b_4)=\frac{q^N\theta(q)}{\theta(q^{N+1})} R_1\big(q^NZ;q^{-N}b_1,b_2,b_3,b_4\big)R_N(Z;qb_1,b_2,b_3,b_4). \end{gather*}
Assuming \eqref{gri} for f\/ixed $N$, it follows that
\begin{gather*}R_{N+1}(Z;b_1,b_2,b_3,b_4)=\frac{q^N\theta(q)}{\theta(q^{N+1})} R_1\big(q^NZ;q^{-N}b_1,b_2,b_3,b_4\big)\\
\qquad{} \times \sum_{\substack{x_1,\dots,x_n\geq 0,\\x_1+\dots+x_n=N}}
\frac{\Delta(zq^x)}{\Delta(z)}\prod\limits_{1\leq i<j\leq n}q^{x_ix_j}(z_iz_j)_{x_i+x_j}\prod\limits_{i=1}^n
\frac{(qz_ib_1)_{x_i}\prod\limits_{j=2}^4(z_ib_j)_{x_i}}{z_i^{x_i}\prod\limits_{j=1}^n(qz_i/z_j)_{x_i}},
\end{gather*}
where $Z=z_1\dotsm z_n$ and $q^NBZ^2=1$. We pull the factor $R_1$ inside the sum and expand it using~\eqref{gri}, with~$z_i$ replaced by $q^{x_i}z_i$. This gives
\begin{gather*}R_{N+1}(Z;b_1,b_2,b_3,b_4)\\
\qquad{} =\frac{q^N\theta(q)}{\theta(q^{N+1})}\sum_{\substack{x_1,\dots,x_n\geq 0,\\x_1+\dots+x_n=N}}\sum_{\substack{y_1,\dots,y_n\geq 0,\\y_1+\dots+y_n=1}}
\frac{\Delta(zq^{x+y})}{\Delta(z)}\prod\limits_{1\leq i<j\leq n}q^{x_ix_j}(z_iz_j)_{x_i+x_j+y_i+y_j}\\
\qquad\quad{}\times\prod\limits_{i=1}^n
\frac{(qz_ib_1)_{x_i}\big(q^{x_i-N}z_ib_1\big)_{y_i}\prod\limits_{j=2}^4(z_ib_j)_{x_i+y_i}}{z_i^{x_i}(z_iq^{x_i})^{y_i}
\prod\limits_{j=1}^n(qz_i/z_j)_{x_i}\big(q^{1+x_i-x_j}z_i/z_j\big)_{y_i}}\\
\qquad{} =\frac{q^N\theta(q)}{\theta(q^{N+1})}\sum_{\substack{x_1,\dots,x_n\geq 0,\\x_1+\dots+x_n=N+1}}\sum_{\substack{y_1,\dots,y_n\geq 0,\\y_1+\dots+y_n=1}} \frac{\Delta(zq^{x})}{\Delta(z)}\prod\limits_{1\leq i<j\leq n}q^{x_ix_j-x_iy_j-x_jy_i}(z_iz_j)_{x_i+x_j}\\
\qquad\quad{}\times\prod\limits_{i=1}^n
\frac{(qz_ib_1)_{x_i-y_i}\big(q^{x_i-y_i-N}z_ib_1\big)_{y_i}\prod\limits_{j=2}^4(z_ib_j)_{x_i}}{z_i^{x_i}q^{(x_i-y_i)y_i}
\prod\limits_{j=1}^n(qz_i/z_j)_{x_i-y_i}\big(q^{1+x_i-x_j-y_i+y_j}z_i/z_j\big)_{y_i}},
\end{gather*}
where we replaced each $x_i$ by $x_i-y_i$ and used that $y_iy_j=0$ for $i\neq j$.

By elementary manipulations, using
\begin{gather*}\prod\limits_{i<j}q^{x_iy_j+x_jy_i}\prod\limits_i q^{(x_i-y_i)y_i}=q^{(x_1+\dots+x_n)(y_1+\dots+y_n)-(y_1^2+\dots+y_n^2)}=q^{(N+1)\cdot 1-1}=q^N, \end{gather*}
the expression above can be rewritten
\begin{gather*}R_{N+1}(Z;b_1,b_2,b_3,b_4)\\
=\frac{1}{\theta(q^{N+1})}\sum_{\substack{x_1,\dots,x_n\geq 0,\\x_1+\dots+x_n=N+1}}
\frac{\Delta(zq^{x})}{\Delta(z)}\prod\limits_{1\leq i<j\leq n}q^{x_ix_j}(z_iz_j)_{x_i+x_j}
\prod\limits_{i=1}^n\frac{(qz_ib_1)_{x_i}\prod\limits_{j=2}^4(z_ib_j)_{x_i}}{z_i^{x_i}\prod\limits_{j=1}^n(qz_i/z_j)_{x_i}}
\\
\qquad{} \times\sum_{\substack{y_1,\dots,y_n\geq 0,\\y_1+\dots+y_n=1}}\prod\limits_{i=1}^n
\frac{\big(q^{x_i-y_i-N}z_ib_1\big)_{y_i}\prod\limits_{j=1}^n\big(q^{1+x_i-y_i}z_i/z_j\big)_{y_i}}{\big(q^{1+x_i-y_i}z_ib_1\big)_{y_i}
\prod\limits_{j=1,\,j\neq i}^n\big(q^{x_i-x_j}z_i/z_j\big)_{y_i}}.
\end{gather*}
Writing $y_i=\delta_{ik}$, the inner sum takes the form
\begin{gather*}\sum_{k=1}^n\frac{\theta\big(q^{x_k-N-1}z_kb_1\big)\prod\limits_{j=1}^n\theta(q^{x_k}z_k/z_j)}{\theta(q^{x_k}z_kb_1)\prod\limits_{j=1,\,j\neq k}^n\theta(q^{x_k-x_j}z_k/z_j)}. \end{gather*}
By \eqref{tpf}, this can be evaluated as{\samepage
\begin{gather*}
\theta\big(q^{N+1}\big)\prod\limits_{i=1}^n\frac{\theta(z_ib_1)}{\theta(q^{x_i}z_ib_1)}=\theta\big(q^{N+1}\big)
\prod\limits_{i=1}^n\frac{(z_ib_1)_{x_i}}{(qz_ib_1)_{x_i}}\end{gather*}
and we arrive at \eqref{gri} with $N$ replaced by $N+1$. This completes the proof of Theorem~\ref{grt}.}

We will now rewrite \eqref{gri} in a way that hides some of its symmetry but makes it clear that it generalizes the Frenkel--Turaev summation \eqref{js}. To this end, we replace $n$ by $n+1$, $z_{n+1}$ by~$q^{-N}a^{-1}$ and eliminate $x_{n+1}$ from the summation. After routine simplif\/ication, we arrive at the following identity.

\begin{Corollary}\label{rsc}
Assuming $a^2q^{N+1}=b_1b_2b_3b_4z_1^2\dotsm z_n^2$,
\begin{gather}\sum_{\substack{x_1,\dots,x_n\geq 0,\\x_1+\dots+x_n\leq N}}
\frac{\Delta(zq^x)}{\Delta(z)}\prod\limits_{i=1}^n\frac{\theta(az_iq^{|x|+x_i})}{\theta(az_i)}
\frac{\prod\limits_{1\leq i<j\leq n}(z_iz_j)_{x_i+x_j}}{\prod\limits_{i=1}^n(aq/z_i)_{|x|-x_i}}
\frac{\big(q^{-N}\big)_{|x|}\prod\limits_{i=1}^n(az_i)_{|x|}}{\prod\limits_{j=1}^4(aq/b_j)_{|x|}} q^{|x|}\nonumber\\
 \qquad\quad{} \times\prod\limits_{i=1}^n
\frac{\prod\limits_{j=1}^4(z_ib_j)_{x_i}}{\big(aq^{N+1}z_i\big)_{x_i}\prod\limits_{j=1}^n(qz_i/z_j)_{x_i}}\nonumber\\
\qquad{} =\frac{\prod\limits_{j=1}^n(aqz_j)_N}{\big(aq/b_1,aq/b_2,aq/b_3,aq/b_1b_2b_3Z^2\big)_N\prod\limits_{j=1}^n(aq/z_j)_N}\nonumber\\
\qquad\quad{} \times\begin{cases}(aq/Z,aq/b_1b_2Z,aq/b_1b_3Z,aq/b_2b_3Z)_N,& n~\text{odd},\\[4mm]
(aq/b_1Z,aq/b_2Z,aq/b_3Z,aq/b_1b_2b_3Z)_N,& n~\text{even}.
\end{cases}\label{rci}
\end{gather}
\end{Corollary}

\section{Applications}\label{ts}

The elliptic Bailey transformation \eqref{bat} can be derived from the elliptic Jackson summation~\eqref{js}. Similar arguments can be used
in multivariable situations, see, e.g.,~\cite{b,bs,mn} for the trigonometric and \cite{r} for the elliptic case.
We will use this method to derive a new multivariable elliptic Bailey transformation by combining the two
multivariable elliptic Jackson summations \eqref{rsi} and \eqref{rci}.

\begin{Theorem}\label{bt}
Suppose that $a^3q^{N+2}=bcdefgz_1^2\dotsm z_n^2$ and let $\lambda=a^2q/bcd$. Then,
\begin{gather}
\sum_{\substack{x_1,\dots,x_n\geq 0,\\x_1+\dots+x_n\leq N}}
\frac{\Delta(zq^x)}{\Delta(z)}\prod\limits_{i=1}^n\frac{\theta(az_iq^{|x|+x_i})}{\theta(az_i)}
\frac{\prod\limits_{1\leq i<j\leq n}(z_iz_j)_{x_i+x_j}}{\prod\limits_{i=1}^n(aq/z_i)_{|x|-x_i}}\nonumber\\
\qquad\quad{}\times\frac{\big(q^{-N},b\big)_{|x|}\prod\limits_{i=1}^n(az_i)_{|x|}}{(aq/c,aq/d,aq/e,aq/f,aq/g)_{|x|}} q^{|x|}\prod\limits_{i=1}^n
\frac{(cz_i,dz_i,ez_i,fz_i,gz_i)_{x_i}}{\big(aq^{N+1}z_i,aqz_i/b\big)_{x_i}\prod\limits_{j=1}^n(qz_i/z_j)_{x_i}}\nonumber\\
\qquad{} =\frac{z^N\prod\limits_{i=1}^n(aqz_i)_N}{(\lambda q,aq/e,aq/f,aq/g)_{N}\prod\limits_{i=1}^n(aq/z_i)_N}\nonumber\\
 \qquad\quad{} \times\sum_{\substack{x_1,\dots,x_n\geq 0,\\x_1+\dots+x_n\leq N}}
\frac{\Delta(zq^x)}{\Delta(z)}\frac{\theta\big(\lambda q^{2|x|}\big)}{\theta(\lambda)}
\frac{\prod\limits_{1\leq i<j\leq n}(z_iz_j)_{x_i+x_j}}{\prod\limits_{i=1}^n(\lambda b/az_i)_{|x|-x_i}}\nonumber\\
 \qquad\quad{} \times\frac{\big(\lambda,q^{-N},\lambda c/a,\lambda d/a\big)_{|x|}\prod\limits_{i=1}^n(\lambda b/az_i)_{|x|}}{\big(\lambda q^{N+1},aq/c,aq/d\big)_{|x|}}q^{|x|}\prod\limits_{i=1}^n\frac{\big(ez_i,fz_i,gz_i,q^{-N}z_i/a\big)_{x_i}}{(aqz_i/b)_{x_i}
 \prod\limits_{j=1}^n(qz_i/z_j)_{x_i}}\nonumber\\
 \qquad\quad{}\times\begin{cases}\displaystyle\left(\frac{a}\lambda\right)^N\frac{(aq/Z,\lambda q/eZ,\lambda q/fZ,\lambda q/gZ)_N}{(q^{-N}Z/a,\lambda q/eZ,\lambda q/fZ,\lambda q/gZ)_{|x|}}, & n~\text{odd},\\[4mm]
\displaystyle\frac{(\lambda q/Z,a q/eZ,a q/fZ,a q/gZ)_N}{(\lambda q/Z,\lambda q/efZ,\lambda q/egZ,\lambda q/fgZ)_{|x|}}, & n~\text{even}.
\end{cases}
\label{bi}\end{gather}
\end{Theorem}

\begin{proof}
If we substitute
 \begin{gather*}(N_1,\dots,N_n,a,b,c,d,e)\mapsto
\big(x_1,\dots,x_n,\lambda,\lambda c/a,\lambda d/a,\lambda b/a,aq^{|x|}\big)\end{gather*}
 in \eqref{rsi}, the right-hand side takes the form
\begin{gather*}\frac{(\lambda q,b)_{|x|}}{(aq/c,aq/d)_{|x|}}\prod\limits_{i=1}^n\frac{(cz_i,dz_i)_{x_i}}{(aqz_i/b,az_i/\lambda)_{x_i}}. \end{gather*}
Thus, the left-hand side of \eqref{bi} can be expressed as
\begin{gather*}
\sum_{\substack{x_1,\dots,x_n\geq 0,\\x_1+\dots+x_n\leq N}}
\frac{\Delta(zq^x)}{\Delta(z)}\prod\limits_{i=1}^n\frac{\theta(az_iq^{|x|+x_i})}{\theta(az_i)}
\frac{\prod\limits_{1\leq i<j\leq n}(z_iz_j)_{x_i+x_j}}{\prod\limits_{i=1}^n(aq/z_i)_{|x|-x_i}}\\
\qquad{} \times\frac{\big(q^{-N}\big)_{|x|}\prod\limits_{i=1}^n(az_i)_{|x|}}{(\lambda q,aq/e,aq/f,aq/g)_{|x|}} q^{|x|}\prod\limits_{i=1}^n
\frac{(az_i/\lambda,ez_i,fz_i,gz_i)_{x_i}}{\big(aq^{N+1}z_i\big)_{x_i}\prod\limits_{j=1}^n(qz_i/z_j)_{x_i}}\\
\qquad{}\times
\sum_{y_1,\dots,y_n=0}^{x_1,\dots,x_n}\frac{\Delta(zq^y)}{\Delta(z)}\frac{\theta\big(\lambda q^{2|y|}\big)}{\theta(\lambda)}
\frac{(\lambda,\lambda c/a,\lambda d/a)_{|y|}\prod\limits_{i=1}^n(\lambda b/az_i)_{|y|}}{\big(aq/c,aq/d,\lambda q^{|x|+1}\big)_{|y|}\prod\limits_{i=1}^n\big(\lambda q^{1-x_i}/az_i\big)_{|y|}} q^{|y|}\\
\qquad{} \times\prod\limits_{i=1}^n\frac{(\lambda q/az_i)_{|y|-y_i}\big(aq^{|x|}z_i\big)_{y_i}\prod\limits_{j=1}^n(q^{-x_j}z_i/z_j)_{y_i}}{(\lambda b/az_i)_{|y|-y_i}(aqz_i/b)_{y_i}\prod\limits_{j=1}^n(qz_i/z_j)_{y_i}}.
\end{gather*}
We change the order of summation and replace the vector $x$ by $x+y$. Some elementary manipulation, using in particular~\eqref{etf}, gives
\begin{gather*}
\sum_{\substack{y_1,\dots,y_n\geq 0,\\y_1+\dots+y_n\leq N}}
\frac{\Delta(zq^y)}{\Delta(z)}\frac{\theta\big(\lambda q^{2|y|}\big)}{\theta(\lambda)}
\frac{\prod\limits_{1\leq i<j\leq n}q^{-y_iy_j}(z_iz_j)_{y_i+y_j}}{\prod\limits_{i=1}^n(aq/z_i,\lambda b/az_i)_{|y|-y_i}}\frac{\prod\limits_{i=1}^n(aqz_i)_{|y|+y_i}}{(\lambda q)_{2|y|}}\\
\qquad{} \times\frac{\big(q^{-N},\lambda,\lambda c/a,\lambda d/a\big)_{|y|}\prod\limits_{i=1}^n(\lambda b/az_i)_{|y|}}{(aq/c,aq/d,aq/e,aq/f,aq/g)_{|y|}}\left(\frac{aq}{\lambda}\right)^{|y|}
\prod\limits_{i=1}^n\frac{(ez_i,fz_i,gz_i)_{y_i}z_i^{y_i}}{\big(aq^{N+1}z_i,aqz_i/b\big)_{y_i}\prod\limits_{j=1}^n(qz_i/z_j)_{y_i}}\\
\qquad{} \times\sum_{\substack{x_1,\dots,x_n\geq 0,\\x_1+\dots+x_n\leq N-|y|}}
\frac{\Delta(zq^{y+x})}{\Delta(zq^y)}\prod\limits_{i=1}^n\frac{\theta\big(az_iq^{|y|+y_i+|x|+x_i}\big)}{\theta\big(az_iq^{|y|+y_i}\big)}\frac{\prod\limits_{1\leq i<j\leq n}\big(z_iz_jq^{y_i+y_j}\big)_{x_i+x_j}}{\prod\limits_{i=1}^n\big(aq^{|y|+1-y_i}/z_i\big)_{|x|-x_i}}\\
\qquad{} \times\frac{\big(q^{|y|-N}\big)_{|x|}\prod\limits_{i=1}^n\big(az_iq^{|y|+y_i}\big)_{|x|}}{\big(aq^{|y|+1}/e,aq^{|y|+1}/f,aq^{|y|+1}/g,\lambda q^{2|y|+1}\big)_{|x|}} q^{|x|}\\
\qquad{}\times\prod\limits_{i=1}^n\frac{\big(ez_iq^{y_i},fz_iq^{y_i},gz_iq^{y_i},q^{y_i-|y|}az_i/\lambda\big)_{x_i}}
{\big(aq^{N+1+y_i}z_i\big)_{x_i}\prod\limits_{j=1}^n\big(q^{1+y_i-y_j}z_i/z_j\big)_{x_i}}.
\end{gather*}
We observe that the inner sum is as in Corollary \ref{rsc}, with the substitutions
\begin{gather*}(z_1,\dots,z_n,N,a,b_1,b_2,b_3,b_4)\mapsto \big(z_1q^{y_1},\dots,z_nq^{y_n},N-|y|,aq^{|y|},e,f,g,q^{-|y|}a/\lambda\big). \end{gather*}
When $n$ is odd, the value of this sum can be rewritten
\begin{gather*}
\frac{(aq/Z,aq/efZ,aq/egZ,aq/fgZ)_{N-|y|}\prod\limits_{i=1}^n\big(aq^{|y|+1+y_i}z_i\big)_{N-|y|}}
{\big(aq^{|y|+1}/e,aq^{|y|+1}/f,aq^{|y|+1}/g,q^{-N-|y|}/\lambda\big)_{N-|y|}\prod\limits_{i=1}^n\big(aq^{|y|+1-y_i}/z_i\big)_{N-|y|}}\\
\qquad{} =z^N\left(\frac{a}{\lambda}\right)^{N-|y|}\frac{(aq/Z,\lambda q/eZ,\lambda q/fZ,\lambda q/gZ)_N\prod\limits_{i=1}^n(aqz_i)_N}{(\lambda q,aq/e,aq/f,aq/g)_{N}\prod\limits_{i=1}^n(aq/z_i)_N}\\
\qquad\quad{} \times\frac{q^{\sum\limits_{i<j}y_iy_j}(\lambda q)_{2|y|}(aq/e,aq/f,aq/g)_{|y|}\prod\limits_{i=1}^n(aq/z_i)_{|y|-y_i}\big(aq^{N+1}z_i,q^{-N}z_i/a\big)_{y_i}}{\big(\lambda q^{N+1},q^{-N}Z/a,\lambda q/eZ,\lambda q/fZ,\lambda q/gZ\big)_{|y|}\prod\limits_{i=1}^n z_i^{y_i}(aqz_i)_{|y|+y_i}},
\end{gather*}
which leads to the right-hand side of \eqref{bi}. The case of even $n$ is treated similarly.
\end{proof}

One may obtain further transformation formulas by iterating Theorem~\ref{bt}. We will only give one example, exploiting the fact that the left-hand side of \eqref{bi} is invariant under interchan\-ging~$c$ and~$e$. In the identity expressing the corresponding symmetry of the right-hand side, we make the substitutions $(\lambda,a,b,c,d)\mapsto(a,a^2q/bcd,aq/cd,aq/bd,aq/bc)$, keeping $e,f,g,z_1,\dots,z_n$ f\/ixed. This leads to another multivariable elliptic Bailey transformation.

\begin{Corollary}\label{bc}
Suppose that $a^3q^{N+2}=bcdefgz_1^2\dotsm z_n^2$ and let $\lambda=a^2q/bde$. Then,
\begin{gather*}
\sum_{\substack{x_1,\dots,x_n\geq 0,\\x_1+\dots+x_n\leq N}}
\frac{\Delta(zq^x)}{\Delta(z)}\frac{\theta\big(a q^{2|x|}\big)}{\theta(a)}
\frac{\prod\limits_{1\leq i<j\leq n}(z_iz_j)_{x_i+x_j}}{\prod\limits_{i=1}^n(b/z_i)_{|x|-x_i}}\\
\quad\times\frac{\big(a,q^{-N},c,d\big)_{|x|}\prod\limits_{i=1}^n(b/z_i)_{|x|}}{\big(a q^{N+1},aq/c,aq/d\big)_{|x|}}q^{|x|}\prod\limits_{i=1}^n\frac{\big(ez_i,fz_i,gz_i,aqz_i/efgZ^2\big)_{x_i}}{(aqz_i/b)_{x_i}\prod\limits_{j=1}^n(qz_i/z_j)_{x_i}}\\
\quad\times\begin{cases}\displaystyle\frac{1}{(a q/eZ,a q/fZ,a q/gZ,aq/efgZ)_{|x|}}, & n~\text{odd},\\[4mm]
\displaystyle\frac{1}{(aq/Z,a q/efZ,a q/egZ,a q/fgZ)_{|x|}}, & n~\text{even}\end{cases}\\
=\frac{(aq,\lambda q/c)_N}{(\lambda q,aq/c)_N}\sum_{\substack{x_1,\dots,x_n\geq 0,\\x_1+\dots+x_n\leq N}}
\frac{\Delta(zq^x)}{\Delta(z)}\frac{\theta\big(\lambda q^{2|x|}\big)}{\theta(\lambda)}
\frac{\prod\limits_{1\leq i<j\leq n}(z_iz_j)_{x_i+x_j}}{\prod\limits_{i=1}^n(\lambda b/az_i)_{|x|-x_i}}\\
\quad\times\frac{\big(\lambda,q^{-N},c,\lambda d/a\big)_{|x|}\prod\limits_{i=1}^n(\lambda b/az_i)_{|x|}}{\big(\lambda q^{N+1}, \lambda q/c,aq/d\big)_{|x|}}q^{|x|}\prod\limits_{i=1}^n\frac{\big(\lambda ez_i/a,fz_i,gz_i,aqz_i/efgZ^2\big)_{x_i}}{(aqz_i/b)_{x_i}\prod\limits_{j=1}^n(qz_i/z_j)_{x_i}}\\
\quad\times\begin{cases}\displaystyle\frac{(aq/cfZ,\lambda q/fZ)_N}{(aq/fZ,\lambda q/cfZ)_N(aq/eZ,\lambda q/fZ,\lambda q/gZ,aq/efgZ)_{|x|}}, & n~\text{odd},\\[4mm]
\displaystyle\frac{(a q/cZ,\lambda q/Z)_N}{(aq/Z,\lambda q/cZ)_N(\lambda q/Z,a q/efZ,a q/egZ,\lambda q/fgZ)_{|x|}}, & n~\text{even}.
\end{cases}
\end{gather*}
\end{Corollary}

Theorem \ref{bt} reduces to Corollary~\ref{rsc} when $aq=bc$. More interestingly, when $b=1$ the left-hand side of~\eqref{bi} reduces to $1$. After a change of parameters, this leads to the following new multivariable elliptic Jackson summation.

\begin{Corollary}\label{njc}
If $a^2q^{N+1}=bcdez_1^2\dotsm z_n^2$, then
\begin{gather*}
\sum_{\substack{x_1,\dots,x_n\geq 0\\x_1+\dots+x_n\leq N}}\frac{\Delta(zq^x)}{\Delta(z)}
\frac{\theta\big(aq^{2|x|}\big)}{\theta(a)}\frac{\prod\limits_{1\leq i<j\leq n}(z_iz_j)_{x_i+x_j}}{\prod\limits_{i=1}^n(e/z_i)_{|x|-x_i}}\frac{\big(a,q^{-N}\big)_{|x|}\prod\limits_{i=1}^n(e/z_i)_{|x|}}{\big(aq^{N+1}\big)_{|x|}} q^{|x|}\\
\qquad\quad{} \times\prod\limits_{i=1}^n\frac{\big(bz_i,cz_i,dz_i,q^{-N}ez_i/a\big)_{x_i}}{(aqz_i/e)_{x_i}\prod\limits_{j=1}^n(qz_i/z_j)_{x_i}}\cdot\begin{cases}
\displaystyle \frac 1{\big(q^{-N}eZ/a,aq/bZ,aq/cZ,aq/dZ\big)_{|x|}}, & n~\text{odd},\\[4mm]
\displaystyle \frac 1{(aq/Z,aq/bcZ,aq/bdZ,aq/cdZ)_{|x|}}, & n~\text{even}
\end{cases}\\
\qquad {} =\frac{(aq,aq/be,aq/ce,aq/de)_{N}\prod\limits_{i=1}^n(aq/ez_i)_N}{Z^N\prod\limits_{i=1}^n(aqz_i/e)_N}\\
\qquad\quad{}\times\begin{cases}
\displaystyle \frac {e^N}{(aq/bZ,aq/cZ,aq/dZ,aq/eZ)_{N}}, & n~\text{odd},\\[4mm]
\displaystyle \frac 1{(aq/Z,aq/beZ,aq/ceZ,aq/deZ)_{N}}, & n~\text{even}.
\end{cases}\end{gather*}
\end{Corollary}

\looseness=-1 Examining the proof of Theorem \ref{bt}, we see that Corollary~\ref{njc} is obtained by combi\-ning~\eqref{rci} with the special case $aq=bc$ of~\eqref{rsi}, when the right-hand side is equal to $\prod\limits_{i=1}^n\delta_{N_i,0}$. The latter identity can be viewed as a matrix inversion~\cite{rs}, so Corollary~\ref{njc} is an inverted version of Corollary~\ref{rsc}, just as~\eqref{rsi} is an inverted version of the standard $A$-type summation~\cite[Corollary~5.2]{r}.

If we let $\lambda d/a=1$ in Corollary~\ref{bc}, we obtain yet another multivariable elliptic Jackson summation. After a change of parameters, it takes the form
\begin{gather}
\sum_{\substack{x_1,\dots,x_n\geq 0\\x_1+\dots+x_n\leq N}}\frac{\Delta(zq^x)}{\Delta(z)}
\frac{\theta\big(aq^{2|x|}\big)}{\theta(a)}\frac{\prod\limits_{1\leq i<j\leq n}(z_iz_j)_{x_i+x_j}}{\prod\limits_{i=1}^n(t/z_i)_{|x|-x_i}}
\frac{\big(a,q^{-N},b,c\big)_{|x|}\prod\limits_{i=1}^n(t/z_i)_{|x|}}{(aq^{N+1},aq/b,aq/c)_{|x|}} q^{|x|}\nonumber\\
\qquad\quad{}\times\prod\limits_{i=1}^n\frac{\big(dz_i,ez_i,tz_i/deZ^2\big)_{x_i}}{\prod\limits_{j=1}^n(qz_i/z_j)_{x_i}}\cdot
\begin{cases}
\displaystyle\frac{1}{(aq/dZ,aq/eZ,t/Z,t/deZ)_{|x|}}, & n~\text{odd},\\[4mm]
\displaystyle\frac{1}{(aq/Z,aq/deZ,t/dZ,t/eZ)_{|x|}}, & n~\text{even}
\end{cases}\nonumber\\
\qquad{} =\begin{cases}
\displaystyle\frac{(aq,aq/bc,aq/bdZ,aq/cdZ)_N}{(aq/b,aq/c,aq/dZ,aq/bcdZ)_N}, & n~\text{odd},\\[4mm]
\displaystyle\frac{(aq,aq/bc,aq/bZ,aq/cZ)_N}{(aq/b,aq/c,aq/Z,aq/bcZ)_N}, & n~\text{even},
\end{cases}\label{jts}
\end{gather}
valid for $a^2q^{N+1}=bcde z_1^2\dotsm z_n^2$ and $t$ arbitrary. This identity is less novel than Corollary~\ref{njc}, as it can be deduced from Theorem~\ref{grt} in a more direct manner. Indeed, writing the sum as
\begin{gather*}\sum_{k=0}^N\sum_{\substack{x_1,\dots,x_n\geq 0\\x_1+\dots+x_n= k}}(\dotsm), \end{gather*}
the inner sum is computed by Theorem~\ref{grt} and the outer sum by~\eqref{js}. In fact, the same proof gives the following more general result, which reduces to~\eqref{jts} when $(d,e)=(fZ,gZ)$ or $(Z,fgZ)$ if $n$ is odd or even, respectively.

\begin{Corollary}
For parameters subject to $a^2q^{N+1}=bcde$, $fghz_1^2\dotsm z_n^2=t$,
\begin{gather*}
\sum_{\substack{x_1,\dots,x_n\geq 0\\x_1+\dots+x_n\leq N}}\frac{\Delta(zq^x)}{\Delta(z)}
\frac{\theta\big(aq^{2|x|}\big)}{\theta(a)}{}
\frac{\prod\limits_{1\leq i<j\leq n}(z_iz_j)_{x_i+x_j}\big(a,q^{-N},b,c,d,e\big)_{|x|}}{\big(aq^{N+1},aq/b,aq/c,aq/d,aq/e\big)_{|x|}}\,q^{|x|}\\
\qquad\quad{} \times\prod\limits_{i=1}^n\frac{(t/z_i)_{|x|}(fz_i,gz_i,hz_i)_{x_i}}{(t/z_i)_{|x|-x_i}\prod\limits_{j=1}^n(qz_i/z_j)_{x_i}}\cdot
\begin{cases}
\displaystyle\frac{1}{(fZ,gZ,hZ,t/Z)_{|x|}}, & n~\text{odd},\\[4mm]
\displaystyle\frac{1}{(Z,fgZ,fhZ,ghZ)_{|x|}}, & n~\text{even}
\end{cases}\\
\qquad{} =\frac{(aq,aq/bc,aq/bd,aq/cd)_N}{(aq/b,aq/c,aq/d,aq/bcd)_N}.
\end{gather*}
\end{Corollary}

\subsection*{Acknowledgements}
This research is supported by the Swedish Science Research Council (Vetenskapsr{\aa}det). I~would like to thank the anonymous referee for a very careful reading of the manuscript, leading to many improvements.

\pdfbookmark[1]{References}{ref}
\LastPageEnding

\end{document}